\documentclass[10pt]{amsart}

\usepackage[margin = 1.0in]{geometry}

\usepackage{amsmath,amssymb,amsthm}
\usepackage[hidelinks]{hyperref}

\numberwithin{equation}{section}

\newtheorem{theorem}{Theorem}[section]
\newtheorem{proposition}[theorem]{Proposition}
\newtheorem{lemma}[theorem]{Lemma}

\DeclareMathOperator{\dist}{dist}
\DeclareMathOperator{\spanop}{span}
\DeclareMathOperator{\conv}{conv}

\newcommand{\R}{\mathbb R}
\newcommand{\p}{\mathbb P}
\newcommand{\abs}[1]{\left|#1\right|}
\newcommand{\norm}[1]{\left\lVert#1\right\rVert}
\newcommand{\absconv}{\operatorname{absconv}}
\newcommand{\cB}{\mathcal B}
\newcommand{\cD}{\mathcal D}
\newcommand{\cV}{\mathcal V}

\title{A $5/8$ Lower Bound on the Banach--Mazur Distance to the Cross-Polytope}

\author{Omer Friedland}

\address{Institut de Math\'ematiques de Jussieu, Sorbonne Universit\'e, 4 place Jussieu, 75005 Paris, France}

\email{omer.friedland@imj-prg.fr}

\date{}

\begin{document}

\begin{abstract}
Let $\delta>0$ be fixed, let $m = \lceil(1+\delta)n\rceil$, and let $\Gamma$ be an
$n\times m$ matrix with independent standard Gaussian entries. For the Gaussian Gluskin
polytope $G_m = \Gamma(B_1^m)$ we prove
$$
\p\left\{d_{\mathrm{BM}}(G_m,B_1^n)
\ge c_\delta n^{5/8}(\log n)^{-1/8}\right\}
\ge 1-Ce^{-cn}.
$$
Consequently, for all sufficiently large $n$,
$$
R_1(n)\ge c n^{5/8}(\log n)^{-1/8}.
$$
The proof uses an exact coefficient quantization, a simultaneous Gaussian compression in
coefficient quotients, and an inverse-trace estimate for arbitrary inscribed bases.
\end{abstract}



\maketitle

\section{Introduction}

For origin-symmetric convex bodies $K,L\subset\R^n$, their Banach--Mazur distance is
$$
d_{\mathrm{BM}}(K,L) = 
\inf\{\rho\ge1:\exists T\in GL(n), T(L)\subset K\subset\rho T(L)\}.
$$
Throughout, all logarithms are natural. A classical fixed-center problem asks for the order of
$$
R_1(n) = \sup\{d_{\mathrm{BM}}(K,B_1^n):K\subset\R^n
\text{ is an origin-symmetric convex body}\},
$$
where
$$
B_1^n = \conv\{\pm e_1,\ldots,\pm e_n\}
$$
is the cross-polytope. Polarity preserves Banach--Mazur distance:
$$
d_{\mathrm{BM}}(K,L)=d_{\mathrm{BM}}(K^\circ,L^\circ).
$$
Since $(B_1^n)^\circ=B_\infty^n$ and polarity is a bijection on origin-symmetric convex
bodies, the corresponding fixed-center radii for $B_1^n$ and $B_\infty^n$ are equal. Thus
upper bounds stated for the cube apply equally to $R_1(n)$. The current general upper bound is
$$
R_1(n)\le Cn^{2/3},
$$
proved by Friedland, Naor and Youssef \cite{FriedlandNaorYoussef2026}. On the lower side,
Szarek obtained a bound of order $\sqrt n\log n$ by random-matrix methods
\cite{Szarek1990}; Tikhomirov later reached the exponent $5/9$, up to logarithmic factors,
for the Gaussian Gluskin model \cite{Tikhomirov2019}. The author's preceding paper improved
the exponent to $4/7$ \cite{Friedland2026}. In August 2026, two independent improvements to
the polynomial exponent $5/8$ were obtained for the Gaussian model. Hmadi proved a $5/8$
lower bound in the regime $m=n^3$, with logarithmic loss $(\log n)^{-5/8}$
\cite{Hmadi2026}. The first arXiv version of the present paper obtained, also for $m=n^3$,
the stronger logarithmic factor $(\log n)^{-1/4}$. Thus the remaining gap between the
polynomial exponents in the lower and upper bounds is $1/24$. Theorem~\ref{thm:main}
strengthens the result further in three quantitative directions: the number of generators is
only $\lceil(1+\delta)n\rceil$, the logarithmic loss is $(\log n)^{-1/8}$, and the
conclusion holds with exponentially high probability.

Let $\Gamma\in\R^{n\times m}$ have independent standard Gaussian entries and write
$$
G_m = \Gamma(B_1^m) = \absconv\{g_1,\ldots,g_m\},
$$
where $g_1,\ldots,g_m$ are the columns of $\Gamma$. Since $m\ge n$, the Gaussian
matrix $\Gamma$ has rank $n$ almost surely, and therefore $G_m$ is full-dimensional almost
surely.

\begin{theorem}[The $5/8$ bound] \label{thm:main}
For every fixed $\delta>0$ there are $c_\delta>0$ and $n_0(\delta)\in\mathbb N$ such
that the following holds for every $n\ge n_0(\delta)$. Let
$$
m = \lceil(1+\delta)n\rceil.
$$
Then
\begin{align} \label{eq:main}
\p\left\{d_{\mathrm{BM}}(G_m,B_1^n)
\ge c_\delta n^{5/8}(\log n)^{-1/8}\right\}
\ge 1-Ce^{-cn},
\end{align}
where $c,C>0$ are universal constants. Consequently, for all sufficiently large $n$,
$$
R_1(n)\ge c n^{5/8}(\log n)^{-1/8}
$$
for a universal constant $c>0$.
\end{theorem}

In the remainder of the paper, $\delta>0$ is fixed and
$$
m = \lceil(1+\delta)n\rceil.
$$
All assertions below are understood for $n$ sufficiently large, with the threshold allowed to
depend on $\delta$; unless explicitly indicated otherwise, the constants denoted by $c,C$
are universal.

The proof begins with an exact formula for the distance in terms of bases inscribed in $G_m$.
It separates the resulting lower bound into an empirical factor and an inverse-row factor.
The empirical factor is bounded below uniformly on the Euclidean sphere because
$m = \lceil(1+\delta)n\rceil$.

For the inverse-row factor, fix an integer $s$. Every coefficient vector $a\in B_1^m$ is
written exactly as $a = b+u$, where $b$ belongs to a finite $1/s$-grid and has at most $s$
nonzero coordinates, while $u\in B_1^m\cap s^{-1}B_\infty^m$. In every quotient determined
by at most $n-r$ quantized vectors, the two corresponding Gaussian clouds can be compressed
simultaneously onto a subspace of dimension comparable with the quotient dimension. The
finite-cloud compression has exceptional probability $e^{-ck^2}$ in a quotient of dimension
$k$; this is what permits simultaneous control over the exponentially many admissible
quotients. Since $m = \lceil(1+\delta)n\rceil$, the entropy cost is
$$
r^2\gtrsim ns\log n.
$$
A trace identity then gives, simultaneously for every inscribed basis,
$$
\beta_{G_m}\gtrsim
\frac{1}{\sqrt r/n+1/\sqrt{rs}}.
$$
Choosing
$$
s\asymp\sqrt{\frac{n}{\log n}},
\quad
r\asymp n^{3/4}(\log n)^{1/4},
$$
yields the exponent $5/8$ and the logarithmic loss $(\log n)^{-1/8}$.

\section{The distance formula and the empirical factor} \label{sec:bridge}

For a surjective deterministic matrix $\Gamma = [g_1 \cdots g_m]\in\R^{n\times m}$, put
$$
K = \Gamma B_1^m,
\quad
\cB(K) = \{X\in GL(n):Xe_i\in K \text{ for }1\le i\le n\},
$$
$$
\beta(X) = \sum_{i = 1}^n\norm{X^{-*}e_i}_2,
\quad
\beta_K = \inf_{X\in\cB(K)}\beta(X),
$$
and
$$
\mu_\Gamma = 
\inf_{\norm{u}_2 = 1}\frac1m\sum_{j = 1}^m\abs{\langle g_j,u\rangle}.
$$

\begin{lemma}[Distance formula] \label{lem:bridge}
For every surjective deterministic $\Gamma$,
\begin{align} \label{eq:bridge}
d_{\mathrm{BM}}(K,B_1^n) = 
\inf_{X\in\cB(K)}\max_{j\le m}\norm{X^{-1}g_j}_1 \ge
\mu_\Gamma\beta_K.
\end{align}
\end{lemma}

\begin{proof}
The inclusion $XB_1^n\subset K$ is equivalent to $X\in\cB(K)$. Since
$K = \absconv\{g_1,\ldots,g_m\}$, the inclusion $K\subset\rho XB_1^n$ is equivalent to
$$
\max_{j\le m}\norm{X^{-1}g_j}_1\le\rho.
$$
This proves the equality in \eqref{eq:bridge}. If $v_i = X^{-*}e_i$, then
$$
\max_{j\le m}\norm{X^{-1}g_j}_1 \ge
\frac1m\sum_{j = 1}^m\sum_{i = 1}^n\abs{\langle g_j,v_i\rangle} \ge
\mu_\Gamma\sum_{i = 1}^n\norm{v_i}_2.
$$
Taking the infimum over $X\in\cB(K)$ gives the inequality.
\end{proof}

\begin{lemma}[Uniform empirical lower bound] \label{lem:empirical}
There is $a_\delta>0$ such that
\begin{align} \label{eq:empirical}
\p\{\mu_\Gamma\ge a_\delta\}\ge1-Ce^{-cn},
\end{align}
where $c,C>0$ are universal.
\end{lemma}

\begin{proof}
For $u\in S^{n-1}$, set
$$
F(u) = \frac1m\norm{\Gamma^*u}_1.
$$
Fix $0<a\le1$. If $Z$ is a standard Gaussian random variable, then
$$
\mathbb E e^{-|Z|/a} = 
\sqrt{\frac2\pi}\int_0^\infty e^{-t/a-t^2/2} dt \le Ca.
$$
Hence, for every fixed $u\in S^{n-1}$,
\begin{align} \label{eq:fixed-small-ball}
\p\{F(u)<2a\} = \p\left\{\sum_{j = 1}^m|\langle g_j,u\rangle|<2am\right\} \le e^{2m}\left(\mathbb E e^{-|Z|/a}\right)^m \le (C_1a)^m.
\end{align}

There is a universal $K>0$ such that
\begin{align} \label{eq:operator-event}
\p\{\norm{\Gamma}_{\mathrm{op}}>K\sqrt m\}\le e^{-c_1m}.
\end{align}
Indeed, for a fixed $u\in S^{n-1}$, $\norm{\Gamma^*u}_2^2$ has the $\chi_m^2$
distribution. A standard exponential estimate for $\chi_m^2$, followed by a union bound over
a $1/4$-net of $S^{n-1}$ of cardinality at most $9^n\le9^m$, gives
\eqref{eq:operator-event} for a sufficiently large universal $K$.

On the event complementary to \eqref{eq:operator-event},
$$
|F(u)-F(v)| \le
\frac1m\norm{\Gamma^*(u-v)}_1 \le
\frac1{\sqrt m}\norm{\Gamma^*(u-v)}_2 \le
K\norm{u-v}_2.
$$
Let $\mathcal N$ be an $a/K$-net of $S^{n-1}$. Then
$$
|\mathcal N|\le(C_2/a)^n.
$$
If $F(v)\ge2a$ for every $v\in\mathcal N$, then $F(u)\ge a$ for every $u\in S^{n-1}$.
Thus, by \eqref{eq:fixed-small-ball},
\begin{align} \label{eq:empirical-union}
\p\{\mu_\Gamma<a\} \le
e^{-c_1m}+C_3^{m+n}a^{m-n}.
\end{align}
Since $m = \lceil(1+\delta)n\rceil$,
$$
m-n\ge\delta n,
\quad
m+n\le(3+\delta)n
$$
for all sufficiently large $n$. Choose $a = a_\delta\in(0,1)$ sufficiently small that
$$
C_3^{3+\delta}a_\delta^\delta\le e^{-1}.
$$
Then the second term in \eqref{eq:empirical-union} is at most $e^{-n}$, while
$e^{-c_1m}\le e^{-c_1n}$. This proves \eqref{eq:empirical}. The dependence on $\delta$
in the distance constant enters only through $a_\delta$ in this lemma; no uniform lower
bound as $\delta\downarrow0$ is asserted.
\end{proof}

\section{Exact coefficient quantization} \label{sec:quantization}

For an integer $1\le s\le m$, set
$$
\cD_s = \{b\in s^{-1}\mathbb Z^m:\norm{b}_1\le1\},
\quad
U_s = B_1^m\cap s^{-1}B_\infty^m,
$$
and
$$
\cV_s = 
\left\{s^{-1}\sum_{j\in I}\varepsilon_je_j:
|I|\le s, \varepsilon_j\in\{-1,1\}\right\}.
$$

\begin{lemma}[Exact quantization] \label{lem:quantization}
Every $a\in B_1^m$ has a decomposition
$$
a = b+u,
\quad
b\in\cD_s,
\quad
u\in U_s.
$$
Moreover,
\begin{align} \label{eq:quant}
|\cD_s|\le(2m+1)^s,
\quad
U_s = \conv\cV_s,
\quad
|\cV_s|\le(2m+1)^s,
\quad
\max_{v\in\cV_s}\norm{v}_2\le s^{-1/2}.
\end{align}
Every $b\in\cD_s$ has at most $s$ nonzero coordinates.
\end{lemma}

\begin{proof}
For $a = (a_j)_{j = 1}^m\in B_1^m$, define
$$
b_j = s^{-1}\operatorname{sgn}(a_j)\lfloor s|a_j|\rfloor,
\quad
u_j = a_j-b_j.
$$
The signs agree coordinatewise, hence
$$
\norm{b}_1+\norm{u}_1 = \norm{a}_1\le1,
\quad
|u_j|<1/s.
$$
Thus $b\in\cD_s$ and $u\in U_s$.

Every $b\in\cD_s$ can be written as the average of $s$ elements of
$$
\{0,\pm e_1,\ldots,\pm e_m\},
$$
with repetitions allowed. This proves $|\cD_s|\le(2m+1)^s$ and also shows that every such
$b$ has support of cardinality at most $s$.

In each closed orthant, multiplication by $s$ identifies $U_s$ with
$$
\left\{x\in[0,1]^m:\sum_{j = 1}^m x_j\le s\right\}.
$$
Every vertex of this polytope is integral. Indeed, if two coordinates are fractional they can
be perturbed in opposite directions, while a single fractional coordinate can be perturbed
unless the sum constraint is active; in the latter case the integrality of $s$ excludes a
single fractional coordinate. Hence $U_s = \conv\cV_s$. The same averaging argument gives
$|\cV_s|\le(2m+1)^s$, and every $v\in\cV_s$ satisfies
$$
\norm{v}_2^2\le\frac{s}{s^2} = \frac1s.
$$
\end{proof}

\section{Simultaneous Gaussian compression} \label{sec:compression}

\begin{lemma}[Compression of a finite Gaussian cloud] \label{lem:cloud}
There are universal constants $C,c>0$ with the following property. Let $k\ge16$, let
$G:E\to\R^k$ be a standard Gaussian operator, meaning that its matrix in orthonormal
bases has independent $N(0,1)$ entries, and let $\mathcal V\subset E$ be finite with
$$
\log|\mathcal V|\le k,
\quad
\norm{v}_2\le R\quad(v\in\mathcal V).
$$
Except with probability at most $e^{-ck^2}$, there is a subspace $F\subset\R^k$ with
$\dim F\le\lfloor k/16\rfloor$ such that
\begin{align} \label{eq:cloud}
\sup_{v\in\mathcal V}\dist(Gv,F)\le CR\sqrt k.
\end{align}
\end{lemma}

\begin{proof}
The case $R = 0$ is immediate. Put $q = \lfloor k/16\rfloor$. Suppose that no such subspace
exists. Starting with $F_0 = \{0\}$, choose recursively $v_{j+1}\in\mathcal V$ with
$$
\dist(Gv_{j+1},F_j)>CR\sqrt k,
\quad
F_j = \spanop\{Gv_1,\ldots,Gv_j\},
$$
for $0\le j\le q$. This is possible because $\dim F_j\le j\le q$. Thus
$v_1,\ldots,v_{q+1}$ satisfy
$$
\dist(Gv_j,\spanop\{Gv_i:i<j\})>CR\sqrt k
\quad(1\le j\le q+1).
$$

Fix an ordered tuple. Write
$$
v_j = p_j+w_j,
\quad
p_j\in\spanop\{v_i:i<j\},
\quad
w_j\perp\spanop\{v_i:i<j\}.
$$
The vectors $w_1,\ldots,w_{q+1}$ are pairwise orthogonal, and the triangular form of this
decomposition gives
$$
\spanop\{v_1,\ldots,v_j\} =
\spanop\{w_1,\ldots,w_j\}
\quad(1\le j\le q+1).
$$
Also $\norm{w_j}_2\le\norm{v_j}_2\le R$. If $w_j=0$, then the corresponding successive
distance is zero, so such a tuple cannot satisfy the preceding inequalities. Otherwise the
Gaussian vectors $Gw_1,\ldots,Gw_{q+1}$ are independent centered Gaussian vectors with
covariances $\norm{w_j}_2^2I_k$. Moreover,
$$
\dist(Gv_j,\spanop\{Gv_i:i<j\}) =
\dist(Gw_j,\spanop\{Gw_i:i<j\}) \le
\norm{Gw_j}_2.
$$
Consequently, conditional on $Gw_1,\ldots,Gw_{j-1}$, the standard Gaussian norm estimate
gives, for a sufficiently large universal $C$,
$$
\p\left\{
\dist(Gv_j,\spanop\{Gv_i:i<j\})>CR\sqrt k
\middle|
Gw_1,\ldots,Gw_{j-1}
\right\}
\le e^{-4k}.
$$
Successive conditioning bounds the probability for the fixed tuple by $e^{-4k(q+1)}$.
There are at most
$$
|\mathcal V|^{q+1}\le e^{k(q+1)}
$$
ordered tuples, so the union bound gives $e^{-3k(q+1)}\le e^{-ck^2}$. The quadratic
exponent in this probability estimate is essential below, since it permits simultaneous
control of the exponentially many admissible coefficient quotients.
\end{proof}

Fix integers $r,s$. A subspace $B\subset\R^m$ will be called \emph{admissible} if it is
spanned by at most $n-r$ elements of $\cD_s$.

\begin{proposition}[Two-cloud compression] \label{prop:clouds}
There are universal constants $C_1,C,c>0$ such that, for all sufficiently large $n$
(depending only on $\delta$), the following holds. Assume
$$
1\le s\le m,
\quad
16\le r\le n,
$$
and
\begin{align} \label{eq:entropy}
r^2\ge C_1ns\log n.
\end{align}
Then, with probability at least $1-2e^{-cr^2}$, simultaneously for every admissible $B$, if
$k = n-\dim B$, there exists a subspace $H_B\subset(\Gamma B)^\perp$ such that
\begin{align} \label{eq:clouds}
\dim H_B\ge\frac{7k}{8},
\quad
\sup_{b\in B_1^m}\norm{P_{H_B}\Gamma b}_2\le C\sqrt k,
\quad
\sup_{u\in U_s}\norm{P_{H_B}\Gamma u}_2\le C\sqrt{\frac{k}{s}}.
\end{align}
\end{proposition}

\begin{proof}
Since $m = \lceil(1+\delta)n\rceil$ and $\delta$ is fixed, for all sufficiently large
$n$,
$$
\log(2m+1)\le2\log n.
$$
By Lemma~\ref{lem:quantization}, the number of admissible subspaces is at most
$$
(n+1)|\cD_s|^n \le
\exp(Cns\log n).
$$
This family is finite and deterministic: every admissible $B$ is generated by vectors from the
fixed set $\cD_s$, before $\Gamma$ is sampled. Since every admissible $B$ has dimension at
most $n-r<n$, the restriction $\Gamma|_B$ is injective almost surely for each fixed $B$.
Therefore, by finiteness, almost surely
$$
\dim \Gamma B = \dim B
$$
simultaneously for all admissible $B$.

Fix an admissible $B$ and condition on $\Gamma|_B$. Put $k = n-\dim B$. Since $B$ is spanned
by at most $n-r$ vectors, $k\ge r$. The restrictions of a standard Gaussian operator to the
orthogonal subspaces $B$ and $B^\perp$ are independent. Hence, after conditioning on
$\Gamma|_B$, the map
$$
P_{(\Gamma B)^\perp}\Gamma|_{B^\perp}:
B^\perp\longrightarrow(\Gamma B)^\perp
$$
is a standard Gaussian operator between Euclidean spaces, with $k$ output dimensions. After
conditioning on $\Gamma|_B$, both $(\Gamma B)^\perp$ and its orthogonal projection are
deterministic; projecting the independent Gaussian restriction $\Gamma|_{B^\perp}$
therefore gives a standard Gaussian operator into this $k$-dimensional space.

Apply Lemma~\ref{lem:cloud} to the two finite clouds
$$
\{\pm P_{B^\perp}e_j:j\le m\},
\quad
\{P_{B^\perp}v:v\in\cV_s\}.
$$
Their radii are at most $1$ and $s^{-1/2}$, respectively. Their logarithmic cardinalities are
at most $C\log n$ and $Cs\log n$. From \eqref{eq:entropy} and $r\le n$,
$$
s\log n\le\frac{r^2}{C_1n}\le\frac r{C_1}\le\frac k{C_1}.
$$
Taking $C_1$ sufficiently large, the cloud lemma applies to both clouds. Notice that, for every
$z\in\R^m$ and every subspace $H\subset(\Gamma B)^\perp$,
$$
P_H\Gamma z = P_H\Gamma P_{B^\perp}z,
$$
since the $B$-component is mapped into $\Gamma B$ and is killed by $P_H$. In particular,
this identity applies to the vectors $e_j$ and to every $v\in\cV_s$.

Thus, except with conditional probability at most $2e^{-ck^2}$, there are subspaces
$F_1,F_2\subset(\Gamma B)^\perp$, each of dimension at most $k/16$, such that the two clouds
lie within distances $C\sqrt k$ and $C\sqrt{k/s}$ of $F_1$ and $F_2$, respectively.
Set
$$
H_B = (F_1+F_2)^\perp\cap(\Gamma B)^\perp.
$$
Then $\dim H_B\ge7k/8$. Projection onto $H_B$ gives the required bounds on the two finite
clouds. Convexity extends the first bound to $B_1^m$, and
$U_s = \conv\cV_s$ extends the second to $U_s$.

Finally, for each fixed admissible $B$ the failure probability is at most
$2e^{-ck^2}\le2e^{-cr^2}$. Therefore
$$
\p\{\text{ failure for some admissible }B\} \le
\#\{B\text{ admissible}\} 2e^{-cr^2} \le
2\exp(Cns\log n-cr^2) \le
2e^{-c'r^2},
$$
where the last inequality follows from \eqref{eq:entropy} after increasing $C_1$ if
necessary. In particular, this event may subsequently be applied to an admissible subspace
selected from an arbitrary inscribed basis depending on the realization of $\Gamma$.
\end{proof}

\section{The inverse-trace estimate and the proof of the theorem} \label{sec:trace}

\begin{proposition}[Inverse-trace estimate] \label{prop:trace}
Under the assumptions of Proposition~\ref{prop:clouds}, with probability at least
$1-2e^{-cr^2}$,
\begin{align} \label{eq:beta}
\beta_{G_m} \ge
\frac{c}{\sqrt r/n+1/\sqrt{rs}}.
\end{align}
\end{proposition}

\begin{proof}
Work on the event of Proposition~\ref{prop:clouds}. Let $X\in\cB(G_m)$. Since every column
of $X$ belongs to $G_m = \Gamma B_1^m$, there is a matrix $A = [a_1 \cdots a_n]$ with
$a_i\in B_1^m$ such that
$$
X = \Gamma A.
$$
The coefficient vector $a_i$ may be chosen independently for each column; no compatibility
between these representations is required.
Write
$$
a_i = b_i+u_i,
\quad
b_i\in\cD_s,
\quad
u_i\in U_s,
$$
using Lemma~\ref{lem:quantization}, and put
$$
v_i = X^{-*}e_i.
$$
Choose $J\subset\{1,\ldots,n\}$ of cardinality $r$ corresponding to the $r$ smallest numbers
among $\norm{v_1}_2,\ldots,\norm{v_n}_2$. Then
\begin{align} \label{eq:small-rows}
\sum_{i\in J}\norm{v_i}_2 \le
\frac rn\sum_{i = 1}^n\norm{v_i}_2 = 
\frac rn\beta(X).
\end{align}

Let
$$
B = \spanop\{b_i:i\notin J\}.
$$
This is admissible. Put $k = n-\dim B\ge r$, and let $H = H_B$ be supplied by
Proposition~\ref{prop:clouds}. If $x_i = \Gamma a_i$ denotes the $i$-th column of $X$, then
$$
I_n = XX^{-1} = \sum_{i = 1}^n x_iv_i^*.
$$
Taking the trace after applying $P_H$ gives
$$
\dim H = 
\sum_{i = 1}^n\langle P_H\Gamma a_i,v_i\rangle.
$$
For $i\notin J$, one has $b_i\in B$, hence $P_H\Gamma b_i = 0$. Thus the $b_i$-terms
outside $J$ vanish; the $b_i$-terms with $i\in J$ are controlled by the uniform
$\cD_s$-estimate, while the $u_i$-terms are controlled for every $i$ by the $U_s$-estimate.
Therefore, using \eqref{eq:clouds} and \eqref{eq:small-rows},
\begin{align*}
\frac{7k}{8}
&\le
\sum_{i\in J}\abs{\langle P_H\Gamma b_i,v_i\rangle} +
\sum_{i = 1}^n\abs{\langle P_H\Gamma u_i,v_i\rangle} \\
&\le C\sqrt k\sum_{i\in J}\norm{v_i}_2 + C\sqrt{\frac{k}{s}}\sum_{i = 1}^n\norm{v_i}_2 \\
&\le C\sqrt k\left(\frac rn+\frac1{\sqrt s}\right)\beta(X).
\end{align*}
Since $k\ge r$,
$$
\beta(X) \ge
\frac{c\sqrt r}{r/n+s^{-1/2}} = 
\frac{c}{\sqrt r/n+1/\sqrt{rs}}.
$$
Taking the infimum over $X\in\cB(G_m)$ proves \eqref{eq:beta}.
\end{proof}

\begin{proof}[Proof of Theorem~\ref{thm:main}]
Choose a sufficiently large universal constant $D$ and set
\begin{align} \label{eq:parameters}
s = \left\lfloor\sqrt{\frac{n}{\log n}}\right\rfloor,
\quad
r = \left\lceil Dn^{3/4}(\log n)^{1/4}\right\rceil.
\end{align}
For all sufficiently large $n$,
$$
\frac12\sqrt{\frac{n}{\log n}} \le s\le \sqrt{\frac{n}{\log n}},
$$
and
$$
1\le s\le m,
\quad
16\le r\le n.
$$
Moreover, if $D$ is sufficiently large,
$$
ns\log n \le
Cn^{3/2}\sqrt{\log n} \le
\frac{r^2}{C_1}.
$$
Therefore Proposition~\ref{prop:trace} applies. Since
$$
s\asymp\sqrt{\frac{n}{\log n}},
\quad
r\asymp n^{3/4}(\log n)^{1/4},
$$
one has
\begin{align} \label{eq:denominator}
\frac{\sqrt r}{n}+\frac1{\sqrt{rs}} \le C n^{-5/8}(\log n)^{1/8},
\end{align}
and
\begin{align} \label{eq:r-square}
r^2\ge c n^{3/2}\sqrt{\log n}.
\end{align}
Consequently, with probability at least
$$
1-2e^{-c n^{3/2}\sqrt{\log n}},
$$
one has
$$
\beta_{G_m}\ge c n^{5/8}(\log n)^{-1/8}.
$$

By Lemma~\ref{lem:empirical}, with probability at least $1-Ce^{-cn}$,
$$
\mu_\Gamma\ge a_\delta.
$$
Intersecting the two events and applying Lemma~\ref{lem:bridge},
$$
d_{\mathrm{BM}}(G_m,B_1^n) \ge
\mu_\Gamma\beta_{G_m} \ge
c_\delta n^{5/8}(\log n)^{-1/8}.
$$
Since $n^{3/2}\sqrt{\log n}\ge n$ for all sufficiently large $n$, this proves
\eqref{eq:main}.

Finally, take $\delta=1$, so that $m=2n$. The event in \eqref{eq:main} has positive
probability for all sufficiently large $n$. Hence some realization gives, with a universal
constant $c>0$,
$$
R_1(n)\ge c n^{5/8}(\log n)^{-1/8}.
$$
\end{proof}

\end{document}